\documentclass{amsart}
\usepackage{amssymb}
\newtheorem{thm}[subsection]{Theorem}
\newtheorem{defn}[subsection]{Definition}
\newtheorem{lemma}[subsection]{Lemma}

\newtheorem{prop}[subsection]{Proposition}
\theoremstyle{definition}

\newtheorem{rmk}[subsection]{Remark}
\newtheorem{eg}[subsection]{Example}
\numberwithin{equation}{section}

\begin{document}

\title[Nonnegative Moore-Penrose inverses  of unbounded gram operators ]{Nonnegative Moore-Penrose inverses  of unbounded gram operators  }

\author{T. Kurmayya}
\address{Department of Mathematics\\National Institute of Technology\\
Warangal, Telangana\\,
India  506004.}

\email{kurmayya@nitw.ac.in}
\author{G. Ramesh}
\address{Department of Mathematics\\I. I. T. Hyderabad, ODF Estate\\ Yeddumailaram, Telangana \\ India-502 205.}

\email{rameshg@iith.ac.in}

%\thanks{}
\subjclass[2010]{ }

\date{19-09-2014}


\begin{abstract}
In this paper we derive necessary and sufficient conditions for the
nonnegativity of Moore-Penrose inverses of unbounded Gram operators
between real Hilbert spaces. These conditions include statements on
acuteness  of certain closed convex cones. The main result
generalizes the existing result for bounded operators \cite[Theorem
3.6]{kurmisivakumargram}.

\end{abstract}
\maketitle
\section{Introduction}
Monotonicity of Gram matrices and Gram operators has received a lot
of attention in recent years. This has been primarily motivated by
applications in convex optimization problems.

A real square  matrix $T$ is called monotone if $x\geq 0$, whenever
$Tx\geq 0$. Here $x=(x_i)\geq 0$ means that $x_i\geq 0$ for all $i$.
Collatz  \cite{cz} has shown that a matrix is monotone if and only
if  it is invertible and the inverse is nonnegative.
% Certain
%properties of monotone matrices and their applications to iterative
%solutions of system of linear equations arising out of applying
%finite difference approximation to a certain second order boundary
%value problems can be found in \cite{cz}.
%The notion of monotonicity
%has undergone generalizations along several directions. We mention
%only some of the more recent works in the literature, beginning with
%the case of classical inverses.
 Gil
gave sufficient conditions on the entries of an infinite matrix $T$
in order for $T^{-1}$ to be nonnegative
 \cite{gil2}.
 %Peris
% adopted a novel approach where he characterized
%nonnegativity of the inverse in terms of splittings of certain type
%\cite{peris}. For operators over  normed vector lattices, extensions
%of some of the results in \cite{peris}, were reported by Weber
%\cite{weber}, with an ingenious proof of an important result
%\cite[Theorem 1]{weber}.
An extension of the notion of monotonicity to characterize
nonnegativity of generalized inverses in the finite dimensional
 case seems to have been first accomplished by   Mangasarian \cite{mg}.  Berman and Plemmons
 \cite{bp3} made extensive contributions to
nonnegative generalized
 inverses by proposing various notions of monotonicity.
 The book by Berman and Plemmons \cite{bp3} contain numerous examples
of applications of nonnegative generalized inverses that include
Numerical Analysis and linear economic models.

The question of monotonicity and their relationships to
nonnegativity of generalized inverses in the infinite dimensional
setting, have been first taken up by Sivakumar (\cite{s1} and
\cite{s2}). Three other types of operator monotonicity were studied
later by Kulkarni and Sivakumar \cite{ks}.
%Finally, we would like to
%mention that the results of Weber mentioned earlier have been
%extended by Kurmayya and Sivakumar \cite{kurmisivakumarpositivity}
%for the case of the Moore-Penrose inverses of bounded operators on
%Hilbert spaces.
For applications of nonnegative Moore-Penrose
inverses of operators to the solution of linear systems of equations
defined by operators between infinite dimensional spaces, we refer
to Kammerer and Plemmons (\cite[Section 6]{kamerer}).
% We also refer
%the reader  \cite[Theorem 2.15 and Theorem 2.16 ]{s3}, where
%nonnegative left inverses and nonnegative Moore-Penrose inverses,
%respectively have been applied in identifying certain classes of
%Banach spaces that satisfy a specific optimization property.

 There is a well
known result by Cegielski that characterizes nonnegative
invertibility of Gram matrices in terms of obtuseness (or acuteness)
of certain polyhedral cones. (See for instance  \cite[Lemma 1.6]{ce}). The results of Cegielski were generalized by Kurmayya
and Sivakumar \cite{kurmisivakumargram} in two directions; from
finite dimensional real Euclidean spaces to  infinite dimensional
real Hilbert spaces and from classical inverses to Moore-Penrose
inverses.

       In this paper we consider linear operators (not necessarily bounded) between real Hilbert spaces  and
       obtain necessary and sufficient conditions for the nonnegativity of Moore-Penrose inverses
       of Gram operators in terms of acuteness of certain closed convex cones. This can be achieved by taking cones in the
       domain of the Gram operator. Because of this slight modification, we
       observe that there is a slight change in some of the existing results (see Lemmas \ref{uinco} and \ref{acutepositivity}, and the
       condition  (\ref{2nd}) in Theorem  \ref{main}). Our results
       generalizes the existing results due to Kurmayya and
       Sivakumar \cite{kurmisivakumargram} and the related results (See for instance Lemma 1.6, \cite{ce}) in the literature.

The paper is organized as follows. In section 2 we introduce some
basic notations, definitions and results. In section 3, we present
some preliminary results and prove the main theorem. In section 4, we illustrate the main theorem with some examples.

\section{Notations and Preliminary results}
Throughout the article we consider infinite dimensional real Hilbert
spaces which will be denoted by $H, H_1,H_2$ etc . The inner product
and the induced norm are denoted  by  $\langle,\rangle$ and $||.||$
respectively.

A subset $K$ of a Hilbert space $H$ is called cone if, $(i)$\;
$x,y\in K\Rightarrow x+y\in K$ and $(ii)$\;  $x\in K$, $\alpha
\in\mathbb{R}$, $\alpha \geq 0\Rightarrow \alpha x\in K$. For a
subset $K$ of a Hilbert space $H$, the dual of $K$ denoted $K^*$ is
defined as $K^*= \{x\in H \colon \langle x,t\rangle \geq 0,
\text{for all}\; t\in K \}$ and $K^{**}=(K^*)^*.$ Note that in
general, $K^{**}=\overline{K},$ where the bar denotes the closure of
$K.$
%If $H=\mathbb{R}^n$ and $K=\mathbb{R}^n_+$ then
%$K^*=-\mathbb{R}^n_+$ and so $K^{**}=K$. If
%$K=\mathbb{R}^n_+\cap R(B^*)$ for some $m\times n$ real matrix $B$,
%then $K^*=-\mathbb{R}^n_+ +N(B)$, where $N(B)$ denotes the null
%space of the matrix $B$. Again $K^{**}=K$.
If $H=\ell^2$, the Hilbert space of all square summable real
sequences and $K=\ell^2_+=\{x\in \ell^2\colon x_i\geq 0, \forall
i\}$, then $K^*=\ell^2_+$ and hence $K^{**}=\ell^2_+.$ A cone $C$ is
said to be acute if  $\langle x,y\rangle\geq 0,$ for all $x,y\in C$.

Let $T$ be a linear operator with domain $D(T)$, a subspace of $H_1$
and taking values in $H_2$,  then the graph $G(T)$ of $T$ is defined
by  $G(T):={\{(x,Tx):x\in D(T)}\}\subseteq H_1\times H_2$. If $G(T)$
is closed, then $T$ is called a closed operator. If $D(T)$ is dense
in $H_1$, then $T$ is called a densely defined operator. For a
densely defined operator there exists a unique linear operator
$T^*:D(T^*)\rightarrow H_1$, where
\begin{equation*}
D(T^*):={\{y\in H_2: \text{the functional} \; x\rightarrow \langle Tx,y\rangle \, \text{for
all}\, x\in D(T)\,\text{is continuous}}\}
\end{equation*}
and $\langle Tx,y\rangle =\langle x,T^*y\rangle$ for all $x\in D(T)$
and $y\in D(T^*)$. This operator is called the adjoint of $T$.
Note that $T^*$ is always closed whether or not $T$ is closed.

The set of all closed operators between $H_1$ and $H_2$ is denoted
by $\mathcal C(H_1,H_2)$ and $\mathcal C(H):=\mathcal C(H,H)$. By
the closed graph Theorem \cite{rud}, an everywhere defined closed
operator is bounded.  Hence the domain of an unbounded closed
operator is a proper subspace of a Hilbert space. For, $T\in
\mathcal C(H_1,H_2)$, the null space and the range space of
$T$ are denoted by $N(T)$ and $R(T)$ respectively and the space
$C(T):=D(T)\cap N(T)^\bot$ is called the carrier of $T$. In fact,
$D(T)=N(T)\oplus^\bot C(T)$ \cite[page 340]{ben}. For a closed subspace $M$ of $H$, we denote the orthogonal projection on $H$
with range $M$ by $P_M$.

If $T\in \mathcal C(H_1,H_2)$ and $S\in \mathcal C(H_2,H_3)$, then $D(ST)={\{x\in D(T):Tx\in D(S)}\}$ and $(ST)(x)=S(Tx)$ for all $x\in D(ST)$.

If $S$ and $T$ are closed operators  with the property that
$D(S)\subseteq D(T)$ and $Sx=Tx$ for all $x\in D(S)$, then $S$ is
called the restriction of $T$ and $T$ is called an extension of $S$. For the details we refer to \cite{goldberg,kato,rud}.

Next, we recall some of the definitions and important results that
we use throughout the article.
\begin{defn}
For a linear map $T:H_1\longrightarrow H_2$, the operator $T^*T$ is
said to be the Gram operator of $T$.
\end{defn}
\begin{defn}\label{geninv}(Moore-Penrose Inverse)\cite[definition 2, page 339]{ben}
Let $T\in \mathcal C(H_1,H_2)$ be densely defined. Then there exists
a unique densely defined operator $T^\dagger \in \mathcal
C(H_2,H_1)$ with domain $D(T^\dagger)=R(T)\oplus ^\bot R(T)^\bot$
and has the following properties:
\begin{enumerate}
\item $TT^\dagger y=P_{\overline{R(T)}}\;y, \;\text{for all}\;y\in D(T^\dagger)$

\item $T^\dagger Tx=P_{N(T)^\bot} \;x, \;\text{for all}\;x\in D(T)$

\item $N(T^\dagger)=R(T)^\bot$.
\end{enumerate}
This unique operator $T^\dagger$ is called the \textit{Moore-Penrose inverse} of $T$.\\
The following property of $T^\dagger$ is also well known. \noindent
For every $y\in D(T^\dagger)$, let \begin{equation*}
L(y):=\Big\{x\in D(T):
||Tx-y||\leq ||Tu-y||\quad \text{for all} \quad u\in D(T)\Big\}.
\end{equation*}
 Here any $u\in L(y)$ is called a \textit{least square solution} of the operator equation $Tx=y$. The vector
 $x=T^\dagger y\in L(y),\,||T^\dagger y||\leq ||u||\quad \text{for all} \quad u\in L(y)$
 and it is called the  \textit{least square solution of minimal norm}.
 A different treatment of $T^\dagger$ is given in \cite{ben},
 where it is called ``\textit{the Maximal Tseng generalized Inverse}".
\end{defn}

 We have the following equivalent definition:
 \begin{defn}\label{equivalentdefn}
  Let $P:=P_{\overline{R(T)}}$. If $y\in R(T)\oplus^\bot R(T)^\bot$, the equation
  \begin{equation}\label{genoperatoreq}
  Tx=Py
  \end{equation}
 always  has a solution. This solution is called a \textbf{least
 square solution}. If $x\in D(T)$ is a least square solution, then
 \begin{equation*}
 ||Tx-y||^2=||Py-y||^2=\displaystyle \min_{z\in D(T)}\, ||Tz-y||^2.
 \end{equation*}
 The unique vector with the minimal norm among all least square
 solutions, is called the \textit{least square solution of minimal
 norm} of the Equation \ref{genoperatoreq} and is given $x=T^\dagger
 y$.
 \end{defn}

Here we list  the properties of the Moore-Penrose inverse, which we
need to prove our main results.
\begin{thm}\cite[theorem 2, page 341]{ben}\label{propertiesofmpinverse}
Let $T\in \mathcal C(H_1,H_2)$ be densely defined. Then
\begin{enumerate}
\item $D(T^\dagger)=R(T)\oplus^\bot R(T)^\bot, \quad
N(T^\dagger)=R(T)^\bot=N(T^*)$
\item \label{rangeofmpi}$R(T^\dagger)=C(T)$
\item $T^\dagger$ is densely defined and $ T^\dagger \in \mathcal C(H_2,H_1)$
\item $T^\dagger$ is continuous if and only $R(T)$ is closed.

\item $T^{\dagger \dagger}=T$
\item $T^{* \dagger}=T^{\dagger *}$
\item $N(T^{* \dagger})=N(T)$
\item  $(T^*T)^\dagger =T^\dagger T^{*
\dagger}$ \label{grammpinverse}
\item  $(TT^*)^\dagger= T^{*
\dagger}T^\dagger$.
\end{enumerate}
\end{thm}

\begin{prop}\cite{ben}\label{nullspaces} Let $T\in \mathcal C(H_1,H_2)$ be densely
defined. Then
\begin{enumerate}
\item $N(T)=R(T^*)^\bot$ \item  $N(T^*)=R(T)^\bot$ \item
$N(T^*T)=N(T)$  and \item $\overline{R(T^*T)}=\overline {R(T^*)}$.
\end{enumerate}
\end{prop}
\begin{prop}\cite{ben,kato} \label{Minmod}
For a densely defined $T\in \mathcal C(H_1,H_2)$, the following
statements  are equivalent:
\begin{enumerate}
\item  $R(T)$ is closed
\item $R(T^*)$is closed
\item $R(T^*T)$ is closed. In this case, $R(T^*T)=R(T^*)$
\item $R(TT^*)$ is closed. In this case, $R(TT^*)=R(T)$.
\end{enumerate}
\end{prop}
\begin{thm}\cite[Theorem 4.1]{shkgrperturb}\label{representationofmpinverse}
Let $T\in \mathcal C(H_1,H_2)$ be densely defined. Assume that
$R(T)$ is closed. Then
\begin{equation*}
(T^*T)^\dagger T^*\subset T^*(TT^*)^\dagger=T^\dagger.
\end{equation*}
\end{thm}

For more information on generalized inverses we refer to \cite{groetsch77,groetsch2007,nashed}.
%Next lemma gives the relation between $C(T)$ and $N(T)^\bot$.
%
%
%\begin{lemma}\cite[lemma 3.3]{shkgrcgt}
%Let $T\in \mathcal C(H_1,H_2)$ be densely defined. Then
%$\overline{C(T)}=N(T)^\bot$.
%\end{lemma}

\section{Main results}
For proving the main theorem (Theorem \ref{main}) we consider the
following results.

Let $H_1$ and $H_2$ be real Hilbert spaces, $T\in \mathcal
C(H_1,H_2)$ be densely defined with closed range. Let $K$ be a
closed convex cone in $D(T^*T)$ such that $K^*\subset D(T^*T)$. Let
$C=TK$ and $D=(T^\dagger)^*K^*$.

\begin{lemma}\label{uinco}
$u\in C^*\cap D(T^*) \Longrightarrow T^*u\in K^*.$
\end{lemma}
\begin{proof}
Let $u\in C^*\cap D(T^*) \mbox{ and } r\in K.$ Then $ 0 \leq \langle
u,Tr\rangle = \langle T^*u,r\rangle. $
\end{proof}
\begin{lemma} \label{acutepositivity}
The following are equivalent :
\begin{enumerate}
\item $ C^*\cap D(T^*) \cap R(T)$ is acute. \label{acute}
\item For all $x,y\in D(T^*T)$ with $T^*Tx \in K^*, T^*Ty \in K^*$, the inequality $\langle
T^*Tx,y\rangle \geq 0$ holds. \label{positivitygram}
\end{enumerate}
\end{lemma}
\begin{proof}
(\ref{acute}) $\Longrightarrow$ (\ref{positivitygram}): Let $x,y\in
D(T^*T)$ satisfy $ T^*Tx\in K^* ~\mbox{and}~ T^*Ty\in K^* . $ For
$r\in K$, we have $Tr\in C$ and hence
\begin{equation*}
\langle Tx,Tr\rangle = \langle T^*Tx,r\rangle \geq 0.
\end{equation*}
So, $Tx\in C^*$. Similarly, we can show that $Ty\in C^*$. Since $C^*
\cap D(T^*)\cap R(T)$ is acute, we have $0\leq \langle Tx,Ty\rangle
= \langle T^*Tx,y\rangle.$

(\ref{positivitygram}) $\Longrightarrow$ (\ref{acute}): Let $u,v\in
C^*\cap D(T^*)\cap R(T).$ Let $u=Tx$ for some $x\in D(T)$. Since
$u\in D(T^*)$, $T^*u$ is defined. That is $x\in D(T^*T).$ Similarly,
$v=Ty$ for some $y\in D(T^*T).$

Next we show that $\langle u,v\rangle\geq 0.$  Since $u\in C^*$, for
$r\in K$ we have
\begin{equation*}
0\leq \langle Tx,Tr\rangle = \langle T^*Tx,r\rangle.
\end{equation*}
Thus $T^*Tx\in K^*$. With a similar argument, we can conclude that
$T^*Ty\in K^*$. By assumption,
\begin{equation*}
\langle
u,v\rangle=\langle Tx,Ty\rangle=\langle T^*Tx,y\rangle \geq 0.
\end{equation*}
Hence $C^*\cap D(T^*)\cap R(T)$ is acute.
\end{proof}
\begin{lemma}\label{dacute}
$D$ is acute if and only if $\langle r,(T^*T)^\dagger s\rangle \geq
0, \text{for every }\; r,s\in K^*.$
\end{lemma}
\begin{proof}
Let $x,y\in D$. Then $x = (T^\dagger)^*r,\; y=(T^\dagger)^*s$ for
some $r,s\in K^*$. Then $D$ is acute if and only if
\begin{equation*}
0\leq \langle x,y\rangle=\langle
(T^\dagger)^*r,(T^\dagger)^*s\rangle
 =\langle r,T^\dagger(T^\dagger)^*s\rangle=\langle r,(T^*T)^\dagger s\rangle,
 \end{equation*}
 by (\ref{grammpinverse}) of Theorem \ref{propertiesofmpinverse}.
\end{proof}
We are now in a position to prove the main result of this paper.

\begin{thm}\label{main}
Let $T\in \mathcal C(H_1,H_2)$ be densely defined with closed range.
Let $K$ be a closed convex cone in $D(T^*T)$ with $T^\dagger
TK\subseteq K$. Let $C=TK$ and $D=(T^\dagger)^*K^*$. Then the
following conditions are equivalent:
\begin{enumerate}
 \item \label{1st} $(T^*T)^\dagger(K^*)\subseteq K$
\item \label{2nd} $C^*\cap D(T^*)\cap R(T)\subseteq C$
\item \label{3rd} $ D$  is acute
\item \label{4th}$ C^*\cap D(T^*) \cap R(T)$ is acute
\item \label{5th} $ T^*Tx\in P_{R(T^*)}(K^*)\Longrightarrow x\in K$
\item\label{6th} $T^*Tx\in K^*\Longrightarrow x\in K$.
\end{enumerate}

\end{thm}
\begin{proof}
(\ref{1st})$\Longrightarrow$ (\ref{2nd}): Let $u\in C^*\cap
D(T^*)\cap R(T)$. Then $u=Tp$ for some $p\in C(T)$.  Then $T^\dagger
u=T^\dagger Tp=P_{N(T)^\bot}p=p$.
 Since $u\in
D(T^*)$, by Theorem \ref{representationofmpinverse},
$T^{\dagger}u=(T^*T)^{\dagger}T^*u$. Set $z=T^\dagger u$. Then
$Tz=TT^\dagger u=P_{R(T)}u=u$. Also $T^*u\in K^*,$ by Lemma
\ref{uinco}. So by the assumption,
$z=(T^*T)^\dagger T^*u\in K.$ Thus $u\in C$.

(\ref{2nd})$\Longrightarrow$ (\ref{3rd}): Let $x=(T^\dagger)^*u$ and
$y=(T^\dagger)^*v$ with $u,v\in K^*$. Since
\begin{align*}
R((T^\dagger)^*)=R((T^*)^\dagger)=C(T^*)&=D(T^*)\cap N(T^*)^\perp\\
                                 &=D(T^*)\cap \overline{R(T)}\\
                                 &=D(T^*)\cap R(T),
                                 \end{align*}
 $x,y\in
D(T^*)\cap R(T)$. Let $r\in K.$ We have $r'=T^\dagger Tr \in K$ (as
$T^\dagger TK\subseteq K$). Then
\begin{equation*}
\langle x,Tr\rangle=\langle(T^\dagger)^*u,Tr\rangle=\langle
u,T^\dagger Tr\rangle=\langle u,r'\rangle \geq 0.
\end{equation*}
Thus $x\in C^*.$ Since $C^*\cap D(T^*)\cap R(T)\subseteq C,$ we have
$x\in C.$ Thus $x=Tp$ for some $p\in K$.

Finally, with $p'=T^\dagger Tp \in K,$ we have,
\begin{equation*}
\langle x,y\rangle=\langle Tp,(T^\dagger)^*v\rangle =\langle
T^\dagger Tp,v\rangle=\langle p',v\rangle \geq 0.
\end{equation*}
Hence $D$ is acute.

(\ref{3rd})$\Longrightarrow$ (\ref{4th}): Let $x,y$ be such that
$r=T^*Tx\in K^*$ and $s=T^*Ty\in K^*$. Since $D$ is acute, by Lemma
\ref{dacute},
\begin{align*}
0\leq \langle r,(T^*T)^\dagger s\rangle
 &=\langle T^*Tx,(T^*T)^\dagger T^*Ty\rangle \\
 &=\langle x,(T^*T)(T^*T)^\dagger (T^*T)y\rangle\\
 &=\langle x,(T^*T)y\rangle\\
 &=\langle T^*Tx,y\rangle.
 \end{align*}
By Lemma \ref{acutepositivity}, $C^*\cap D(T^*)\cap R(T)$ is acute.

(\ref{4th})$\Longrightarrow$ (\ref{5th}): Let $T^*Tx=P_{R(T^*)}w$
for some $w\in K^*$. Since, $R(T^*T)=R(T^*)$, we have
$T^*Tx=P_{R(T^*T)}w$. Hence $x=(T^*T)^\dagger w$ (By Definition \ref{equivalentdefn} ).
%Then $Tx=T(T^*T)^\dagger w$.
%Thus
%\begin{equation*}
%T^\dagger Tx=T^\dagger T(T^*T)^\dagger
%w=T^\dagger(T^\dagger)^*w.
%\end{equation*}
%If $x\in R(T^*)$, then $T^{\dagger} Tx=P_{N(T)^\perp}x=P_{R(T^*)}x=x$.

Let $r\in K^*.$ Then
\begin{equation*}
\langle x,r\rangle=\langle (T^*T)^\dagger w,r\rangle =\langle
T^\dagger (T^\dagger)^*w,r\rangle=\langle
(T^\dagger)^*w,(T^\dagger)^*r\rangle.
\end{equation*}
Set $u=(T^\dagger)^*w,
~v=(T^\dagger)^*r.$ Then, as was shown earlier, $u,v\in R(T)\cap
D(T^*)$. For $t\in K$, with $t'=T^\dagger Tt \in K,$ we have
\begin{equation*}
\langle u,Tt\rangle=\langle (T^\dagger)^*w,Tt\rangle =\langle
w,T^\dagger Tt\rangle =\langle w,t'\rangle \geq 0.
\end{equation*}
So $u\in C^*$. Along similar lines it can be shown that $v\in C^*$.
Thus for all $r\in K^*, ~\langle x,r\rangle=\langle u,v\rangle\geq
0$. So $x\in (K^*)^*=K.$

(\ref{5th})$\Longrightarrow$ (\ref{6th}): Choose $x$ such that
$T^*Tx\in K^*$. We have \begin{equation*}
T^*Tx=P_{R(T^*T)}(T^*Tx)=P_{R(T^*)}(T^*Tx)\in P_{R(T^*)}(K^*).
\end{equation*}
Hence by  (\ref{5th}), $x\in K$.

(\ref{6th})$\Longrightarrow$ (\ref{1st}): Let $u=(T^*T)^\dagger v$
with $v\in K^*$.
%Since,
%\begin{align*}
%R((T^*T)^\dagger)&=C(T^*T)\;\; (\text{by (\ref{rangeofmpi}) \;\text{of Theorem}\; (\ref{propertiesofmpinverse} }))\\
%&=D(T^*T)\cap
%N(T^*T)^\perp\\
%&=D(T^*T)\cap N(T)^\perp\\
%&=D(T^*T)\cap R(T^*),
%\end{align*}
%$u\in
%R(T^*)$.

Then $T^*Tu=T^*T(T^*T)^\dagger v=P_{R(T^*)}v=T^\dagger Tv.$ Then for
$r\in K$ with $r'=T^\dagger Tr \in K,$ we have
\begin{equation*}
\langle T^*Tu,r\rangle=\langle T^\dagger Tv,r\rangle=\langle
v,T^\dagger Tr\rangle=\langle v,r'\rangle\geq 0.
\end{equation*}
Thus $T^*Tu\in K^*.$~ As (\ref{6th}) holds, $u\in K$. Thus
$(T^*T)^\dagger(K^*)\subseteq K.$

This completes the proof of the theorem.
\end{proof}

%\begin{cor}
%Let $H_1=\mathbb{R}^n,~ H_2=\mathbb{R}^m$ and $K=\mathbb{R}^n_+\cap
%R(T^*).$ Then the conditions (\ref{2nd})-(\ref{6th}) are equivalent to
%$(T^*T)^\dagger(\mathbb{R}^n_+)\subseteq \mathbb{R}^n_+$, that is
%entry wise non-negativity of $(T^*T)^\dagger$.
%\end{cor}
%\begin{proof}
%Clearly, $T^\dagger TK\subseteq K,$ $K^*=-\mathbb{R}^n_++N(T)$
%and $K^{**}=K.$ By Theorem \ref{main}, (\ref{2nd})-(\ref{6th}) are
%equivalent to (\ref{1st}): $(T^*T)^\dagger(\mathbb{R}^n_+{+}N(T))\subseteq
%\mathbb{R}^n_+\cap R(T^*)$. The conclusion now follows, since
%$N((T^*T)^\dagger)=N(T)$ and $R((T^*T)^\dagger)=R(T^*)$.
%\end{proof}
%
%The next corollary includes Cegielski's result \cite[Lemma 1.6]{ce}) as a particular case. Note that if $T$ is of full column
%rank, then $T^\dagger T=I$, so that $T^\dagger TK\subseteq K$ holds
%trivially.
%
%\begin{cor}
%In addition to the conditions of Theorem \ref{main}, suppose that
%$A$ is of full column rank. Then the conditions (\ref{2nd})-(\ref{6th}) are
%equivalent to $(T^*T)^{-1}\geq 0$.
%\end{cor}
\begin{rmk}\hfill
\begin{enumerate}
\item In \cite{kurmisivakumargram}, conditions (5) and (6) were shown to
be equivalent to each other and also equivalent to the nonnegativity
of $(T^*T)^\dagger$ under an assumption that $x\in R(T^*)$ by
Kurmayya and Sivakumar. Here, we make a remark that the above
mentioned assumption is redundant to prove equivalence of those
conditions. If $ T^*Tx\in P_{R(T^*)}(K^*)$ then it can be shown that
$x\in R(T^*)$.
\item If $K\subseteq C(T^*T)$, then the condition $T^{\dagger}TK\subseteq K$ is satisfied automatically.
\item If $T$ is one-to-one, then $T^{\dagger}T=I$ and hence in this case $T^{\dagger}TK\subseteq K$ holds for any cone in $D(T)$.
\end{enumerate}

\end{rmk}

\section{Examples}
In this section, we illustrate Theorem \ref{main} with examples.
\begin{eg}
 Let $H=\ell^2$ and $D(T)=\Big\{(x_1,x_2,\dots)\in H \colon
\displaystyle \sum_{j=1}^{\infty} |jx_j|^2<
 \infty  \Big\}.$ Define $T:D(T)\rightarrow H$ by $$T(x_1,x_2,x_3,\dots,x_n,\dots)=
 (x_1,2x_2,3x_3,\dots,nx_n,\dots) \quad \text{for all}\; (x_1,x_2,\dots)\in D(T).$$
 Since  $D(T)$ contains $ c_{00}$, the space of all sequences having at most finitely many nonzero terms, we have $\overline{D(T)}=H.$ Clearly $T$ is unbounded and closed since $T^*=T$. By \cite[example 5.1]{knr}, $R(T)$ is closed. In fact, $T^{-1}$ exists and
 \begin{equation*}
 T^{-1}(y_1,y_2,y_3,\dots,y_n,\dots)=(y_1,\frac{y_2}{2},\frac{y_3}{3},\dots,\frac{y_n}{n},\dots),\; \text{for all}\;(y_n)\in H.
 \end{equation*}

 Note that $D(T^*T)={\{(x_n)\in H: \displaystyle \sum_{n=1}^\infty n^4|x_n|^4<\infty}\}$. Let
 \begin{equation*}
K={\{(x_n)\in D(T^2):x_n\geq 0\; \text{for all}\; n\in \mathbb N }\}.
\end{equation*}
Clearly, $K^*=K$ and $T^{\dagger }TK=K$. Hence $K$ satisfy the
Hypothesis of Theorem \ref{main}. In this case,
$D=T^{{\dagger}^*}(K^*)=T^{-1}(K)$. Let $x,y\in D$. Then
$x=T^{-1}u,\; y=T^{-1}v$ for some $u,v\in H$. Then Let
$u=\displaystyle \sum_{n=1}^\infty \langle u,e_n\rangle e_n$ and
$v=\displaystyle \sum_{n=1}^\infty \langle v,e_n\rangle e_n$ (Here
${\{e_n:n\in \mathbb N}\}$ is the standard orthonormal basis for
$H$). Then
\begin{align*}
\langle x,y\rangle &=\langle T^{-2}u,v\rangle\geq 0\\
                   &= \displaystyle \sum_{n=1}^\infty \frac{1}{n^2}\langle u,e_n\rangle \, \langle v,e_n\rangle \\
                   &\geq 0 \quad ( \text{since} \; \langle u,e_n\rangle,  \langle v,e_n\rangle\geq 0).
\end{align*}
Therefore $D$ is acute. Hence by Theorem \ref{main},
$(T^*T)^{\dagger}$ is nonnegative with respect to the cone $K$. This
can be easily verified independently by using the definition.
\end{eg}

 \begin{eg}
Let $H=\ell^2$ and $ D(T)=\big\{(x_1,x_2,\dots,x_n,\dots)\colon
\displaystyle \sum_{j=2}^\infty |jx_j|^2<\infty \big\}.$ Define
$T:D(T)\rightarrow H$ by
$$T(x_1,x_2,\dots,x_n,\dots)=\big(0,2x_2,3x_3,4x_4,\dots\big)\quad \text{for all}\quad(x_1,x_2,\dots)\in H.$$
Observe that $T$ is densely defined,  $T=T^*$ and
$N(T)={\{(x_1,0,0,\dots)\colon x_1 \in \mathbb {C}}\}$. Hence
$C(T)=\Big\{(0,x_2,x_3,\dots)\colon
\displaystyle\sum_{j=2}^\infty|jx_j|^2<\infty\Big\}$. We can show
that $R(T)$ is closed (see \cite[example 5.2]{knr} for details) and
\begin{equation*}
T^{\dagger}(y_1,y_2,y_3,\dots,)=\big(0,\frac{y_2}{2},\frac{y_3}{3},\dots\big),\; (y_n)\in \ell^2.
\end{equation*}

It can be seen that $T=T^*$ and $D(T^2)={\{(x_n)\in H: \displaystyle \sum_{n=2}^\infty n^4|x_n|^4<\infty}\}$.
Take \begin{equation*}
K={\{(x_n)\in D(T^2):x_n\geq 0\; \text{for all}\; n=2,3,\dots }\}.
\end{equation*}
 It is easy to verify that $K^*=K$ and $T^{\dagger} TK\subseteq K$. Also $D=T^{{\dagger}^*}(K^*)=T^{\dagger}(K)$. Let $x,y\in D$. Then $x=T^{\dagger}u,\; y=T^{\dagger}v$ for some $u,v\in H$. Then Let $u=\displaystyle \sum_{n=1}^\infty \langle u,e_n\rangle e_n$ and $v=\displaystyle \sum_{n=1}^\infty \langle v,e_n\rangle e_n$. Then
\begin{align*}
\langle x,y\rangle &=\langle T^{\dagger}u,T^{\dagger}v \rangle\geq 0\\
                   &= \displaystyle \sum_{n=2}^\infty \frac{1}{n^2}\langle u,e_n\rangle \, \langle v,e_n\rangle \\
                   &\geq 0 \quad ( \text{since} \; \langle u,e_n\rangle,  \langle v,e_n\rangle\geq 0).
\end{align*}
Therefore $D$ is acute. Hence by Theorem \ref{main}, $(T^*T)^{\dagger}$ is positive with respect to the cone $K$.
\end{eg}

\begin{eg}
Let $\mathcal {AC}[0,\pi]$ denote the space of all absolutely
continuous functions on $[0,\pi]$. Let
\begin{align*}
H&:= \text{The real space} \; L^2[0,\pi]\; \text{of real valued
functions}\\
H'&:=\Big\{ \phi \in \mathcal{AC}[0,\pi]: \phi' \in H \Big\},\\
H''&:={\{\phi\in H':\phi'\in H'}\}.
\end{align*}
Let $L:=\displaystyle \frac{d}{dt}$ with $D(L)={\{x\in H':
\phi(0)=\phi(\pi)=0}\}.$

It can be shown using the fundamental theorem of integral calculus
that $L\in \mathcal C(H)$. Let $\phi_n=\sin(nt),\;n\in \mathbb N$.
Then ${\{\phi_n:n\in \mathbb N}\}$ is an orthonormal basis for $H$
and is contained in $D(L)$, hence $L$ is densely defined. Also
$C(L)=D(L)$. i.e., $L$ is one-to-one. It can be shown that
$R(L)={\{y\in H: \int\limits\limits_0^\pi
y(t)\,{d}t=0}\}=\text{span}\,{\{1}\}^\bot$. Hence in this case
$D(L^\dagger)=H$. Let $\psi_n=\sqrt{\frac{2}{\pi}}\, \cos(nt),\;t\in
[0,\pi],\; n\in \mathbb N$. Then ${\{\psi_n:n\in \mathbb N}\}$ is an
orthonormal basis for $R(L)$.

We have, $L^*L=-\frac{d^2}{dt^2}$ with  $D(L^*L)={\{\phi \in
H'':\phi(0)=0=\phi(\pi) }\}$ \cite[page 349]{ben}.  By using the
projection method (see \cite[example 3.5]{shkgrgeninv}), we can show
that
\begin{equation}\label{repnofmpiofdiff}
L^\dagger(y)=\displaystyle \sum_{n=1}^\infty \frac{1}{n}\langle
y,\psi_n\rangle \phi_n.
\end{equation}

Let $K={\{\phi\in D(L^*L): \langle \phi,\phi_n\rangle \geq 0,\;
\text{for all}\; n\in \mathbb N}\}$. Then $K$ is a cone and $K^*=K$.
We verify condition \ref{1st} of Theorem \ref{main}. First note
that, by Equation \ref{repnofmpiofdiff}, we have

\begin{equation}\label{adjointofmpiofdiff}
L^{\dagger *}\phi=\displaystyle \sum_{n=1}^\infty \frac{1}{n}\langle \phi,\phi_n\rangle \psi_n, \;\; \text{for all}\; \phi \in H.
\end{equation}
Now, let $f\in K$. Then
\begin{align*}(L^*L)^{\dagger}(f)&=L^{\dagger}(L^{\dagger })^*(f)\\
                                 &=\displaystyle \sum_{n=1}^\infty \frac{1}{n^2}\langle f,\phi_n\rangle \phi_n.
\end{align*}
Since $f\in K$, we have $\langle f,\phi_n\rangle\geq 0$ for all
$n\in \mathbb N$ and so $\frac{1}{n^2}\langle f,\phi_n\rangle\geq 0$
for all $n\in \mathbb N$. This concludes that
$(L^*L)^{\dagger}(K^*)\subseteq K$.
%Hence, by Theorem \ref{main}, $(L^*L)^{\dagger}$ is positive with respect to the cone $K$.
\end{eg}

\vspace{.5cm} \noindent \textbf{Acknowledgements}: We thank Prof.K.C.
Sivakumar who studied this article thoroughly and helped us to
improve and bring it in this form.

%The authors thank a referee and Professor A. Cegielski of Zielona
%Gora for suggestions and comments that have improved the readability
%of the article.
\bibliographystyle{amsplain}

\end{document}